\documentclass[10pt]{amsart}
\allowdisplaybreaks

\usepackage{amsfonts}
\usepackage{amsmath}
\usepackage{amssymb}
\usepackage{amsthm}
\usepackage{graphicx} \usepackage{enumerate} \usepackage{multicol}
\usepackage{mathrsfs} \usepackage[all,cmtip]{xy}
\usepackage{color}

\newcounter{dummy} \numberwithin{dummy}{section}
\newtheorem{theorem}[dummy]{Theorem}
\newtheorem{corollary}[dummy]{Corollary}
\newtheorem{lemma}[dummy]{Lemma}
\newtheorem{definition}[dummy]{Definition}
\newtheorem{proposition}[dummy]{Proposition}
\theoremstyle{remark}
\newtheorem{remark}[dummy]{Remark}

\newcommand{\frakn}{\mathfrak{n}}

\newcommand{\scrL}{\mathscr{L}}

\DeclareMathOperator{\Ann}{Ann}
\DeclareMathOperator{\Sym}{Sym}
\DeclareMathOperator{\id}{id}
\DeclareMathOperator{\pr}{pr}
\DeclareMathOperator{\rank}{rank}
\DeclareMathOperator{\spn}{span}

\DeclareMathOperator{\Ad}{Ad}
\DeclareMathOperator{\ad}{ad}
\DeclareMathOperator{\gl}{\mathfrak{gl}}

\DeclareMathOperator{\Hol}{Hol}
\DeclareMathOperator{\Tay}{Tay}
\DeclareMathOperator{\disk}{disk}
\DeclareMathOperator{\Lie}{Lie}
\newcommand{\bnabla}{\boldsymbol \nabla}
\newcommand{\sfd}{\mathsf{d}}
\newcommand{\eps}{\varepsilon}

\newcommand\blank{{\kern.8pt\displaystyle\cdot\kern.8pt}}

\numberwithin{equation}{section}

\title[Asymptotic expansions of holonomy]{Asymptotic expansion of holonomy}
\author[E.~Grong and P.~Pansu]{Erlend Grong and Pierre Pansu}

\address{Universit\'e Paris Sud, Laboratoire des Signaux et Syst\`emes (L2S) Sup\'elec, CNRS, Universit\'e Paris-Saclay, 3 rue Joliot-Curie, 91192
Gif-sur-Yvette, France and University of Bergen, Department of Mathematics, P. O. Box 7803,
5020 Bergen, Norway.}
\email{erlend.grong@gmail.com}

\address{Laboratoire de Math\'ematiques d'Orsay, Universit\'e Paris-Sud, CNRS,
Universit\'e Paris-Saclay, 91405 Orsay, France.}
\email{pierre.pansu@math.u-psud.fr}

\subjclass[2010]{53C29, 41A99, 53C17}

\keywords{asymptotic expansion of holonomy, radial gauge, sub-Riemannian geometry, horizontal holonomy}

\begin{document}
\let\thefootnote\relax\footnotetext{
The first author is supported by the Fonds National de la Recherche Luxembourg (AFR 4736116 and OPEN Project GEOMREV) and by the Research Council of Norway (project number 249980/F20). The second author acknowledges support from ANR, project SRGI, ANR-15-CE40-0018.}

\begin{abstract}
Given a principal bundle with a connection, we look for an asymptotic expansion of the holonomy of a loop in terms of its length. This length is defined relative to some Riemannian or sub-Riemannian structure. We are able to give an asymptotic formula that is independent of choice of gauge.
\end{abstract}

\maketitle

\section{Introduction}
Let $G \to P \to M$ be a principal bundle with some connection $\omega$ and let $\Omega$ be its curvature form. The realization of curvature as the infinitesimal generator of holonomy of $\omega$ appears in the now classical Ambrose and Singer theorem, \cite{AmSi53}. Further elaborated by Ozeki \cite{Oze56}, this result explains how the holonomy group of $\omega$ is infinitesimally generated by $\Omega$ and its covariant derivatives at a point.

If we look at an individual loop $\gamma$ based in $x \in M$, we can also make the correspondence concrete in the case when $G$ is abelian. An application of Stokes theorem shows that for any $p \in P_x$, the corresponding holonomy $\Hol^\omega_p(\gamma)$ equals $\exp(- \int_{\disk(\gamma)} \sigma^* \Omega )$ where $\sigma$ denotes an arbitrary gauge satisfying $\sigma(x) = p$ and $\int_{\disk(\gamma)}$ denotes the integral over an arbitrary filling disk. Using ideas from \cite{Pan93}, we are able to show that this result holds approximately for short loops relative to some Riemannian or even sub-Riemannian metric on $M$. Such expressions have real world applications, see \cite{MuSa93,HaCh10}. A question in these applications if often to look for the best choice of gauge. In the case of $M = \mathbb{R}^n$, we obtain the following expansion using the radial gauge.
\begin{theorem} \label{th:Approx1}
For every $x \in \mathbb{R}^n$, define
$$m_\Omega(x) = \left\{ \begin{array}{ll} 3 + k & \text{$\Omega$ and all its covariant derivatives of order $\leq k$ vanish at $x$}, \\ 2 & \text{otherwise.} \end{array} \right.$$
Let $\sigma: \mathbb{R}^n \to P$ be a gauge in which $\omega$ is parallel in radial directions emanating from $x$. Then for $p = \sigma(x) \in P_x$, there is an $\eps >0$ and $C\geq 0$ such that
\begin{equation} \label{Approx1} \left| \exp^{-1} \Hol_p^\omega(\gamma) + \int_{\disk(\gamma)} \sigma^* \Omega \right| \leq  C \ell(\gamma)^{2m_\Omega(x)},\end{equation}
for every loop $\gamma$ based at $x$ with $\ell(\gamma) < \eps$. In the above formula, $\int_{\disk(\gamma)}$ denotes the integral over a radial filling disk.
\end{theorem}
The above result also holds true on Carnot groups, which also have dilations and hence the possibility of defining a radial gauge. The details are found in Section~\ref{sec:Flat}.

For this reason, the radial gauge is convenient, but may not always be so simple to compute in practice. We therefore continue in Section~\ref{sec:GaugeFree}, showing that the two-form $\sigma^* \Omega$ in \eqref{Approx1} can be replaced by a two-form with polynomial coefficients whose terms are independent of choice of gauge. For the special case of $\mathbb{R}^n$ with coordinates $(z_1, \dots, z_n)$, we have the following result.
\begin{theorem} \label{th:Approx2}
Write $\partial_k = \frac{\partial}{\partial z_k}$ and let $x = (x_1, \dots, x_n)$ be any point. Let $\sigma$ be an arbitrary gauge such that $\sigma(x) = p$. Use the notation $\omega_k^\sigma = (\sigma^* \omega)( \partial_k)$ and $\Omega^\sigma_{ij} = (\sigma^* \Omega)(\partial_i, \partial_j)$ and define
\begin{align} \label{RnF3}
F_p(\gamma) & = \frac{1}{2} \sum_{i,j=1}^n \Omega_{ij}^\sigma(x) \int_{\gamma}  (z_i - x_i) dz_j \\ \nonumber
& \quad + \frac{1}{3} \sum_{i,j,k=1}^n \left( (\partial_k \Omega_{ij}^\sigma)(x) + [ \omega_k^\sigma(x), \Omega_{ij}^\sigma(x) ]\right) \int_\gamma (z_i - x_i) (z_k- x_k) dz_j ,
\end{align}
for any loop based at $x$. Then $F_p$ does not depend on the choice of gauge. Furthermore, there exist constants $\eps > 0$ and $C \geq 0$, such that $\left| \exp^{-1} \Hol_p^\omega(\gamma) + F_p(\gamma) \right| \leq  C \ell(\gamma)^{4}$ for every $\gamma$ with $\ell(\gamma) < \eps$.
\end{theorem}

For a general Riemannian or sub-Riemannian manifold, we first need to choose an appropriate local coordinate system of a type called privileged. However, after this choice is made, the above statement remains true with respect to the dilations in this coordinate system. We give the details of this in Section~\ref{sec:Curved}.

It was observed in \cite{Pan93} that integrating the curvature over a radial filling disk in the sub-Riemannian Heisenberg group will actually give a better approximation of the holonomy than in the Euclidean case. This result was found using ideas from \cite{FGR97,Rum94}. By applying the new theory of horizontal holonomy found in \cite{CGJK15}, we are able to show the following (for more details, see Section~\ref{sec:Equivalent}).
\begin{theorem}
Let $\pi: P \to M$ be a principal bundle over a sub-Riemannian manifold $(M, D, g)$, where $D$ has rank $n_1$. Assume that the sections of $D$ and their iterated brackets up to order $k$ span a subbundle of rank equal to the rank of the free nilpotent algebra of step $k+1$ and of $n_1$ generators. Then there is a``modified curvature'' $\tilde \Omega$ such that
\begin{equation*}  \left| \exp^{-1} \Hol_p^\omega(\gamma) + \int_{\disk(\gamma)} \sigma^* \tilde \Omega \right| \leq  C \ell(\gamma)^{2k+2},\end{equation*}
\end{theorem}

The result of Theorem~\ref{th:Approx2} follows from computations on Euclidean space found in Section~\ref{sec:Examples}. In that Section, we also look at the examples of Riemannian manifolds, the Heisenberg group and the sub-Riemannian Hopf fibration.

\section{Euclidean space and Carnot groups} \label{sec:Flat}

\subsection{Sub-Riemannian manifolds} \label{sec:SR}
Throughout our paper, any manifold $M$ will be connected. In this section we will revisit some basic facts about sub-Riemannian manifolds. For more details, we refer to \cite{Mon02}.

On a manifold $M$, \emph{a sub-Riemannian structure} is a pair $(D, g)$, where $D$ is a subbundle of the tangent bundle $TM$ and $g\in \Gamma(\Sym^2 D^*)$ is a positive definite, smooth, symmetric tensor defined only on $D$. We will refer to the subbundle~$D$ as \emph{the horizontal bundle}. For any $v, w \in D_x$, we write $\langle v, w\rangle_g = g(v,w)$ and $|v|_g = \langle v,v\rangle_g^{1/2}$. We say that a continuous curve $\gamma: [a,b] \to M$ is \emph{horizontal} if is absolutely continuous and satisfies $\dot \gamma(t) \in D_{\gamma(t)}$ for almost every $t \in [a,b]$. For such a curve, we define its length $\ell(\gamma)$ by
$$\ell(\gamma) = \int_a^b | \dot \gamma(t) |_g \, dt.$$
We have the corresponding \emph{Carnot-Carath\'eodory metric} of the sub-Riemannian structure, 
\begin{equation} \label{CC} \sfd(x,y) = \inf \left\{ \ell(\gamma) \, : \, \begin{array}{c} \text{$\gamma:[0,1] \to M$ horizontal} \\ \text{$\gamma(0) = x$, $\gamma(1) = y$} \end{array} \right\}.\end{equation}
The distance in \eqref{CC} is in general not finite for two arbitrary points $x$ and $y$.

For a given horizontal bundle $D$, define $\underline{D}^1 = \Gamma(D)$ and for every $j = 1, 2, \dots$,
$$\underline{D}^{j+1} = \spn \left\{ X, [X,Y] \, : \, X \in \underline{D}^j, Y \in \Gamma(D) \right\}.$$
For any point $x \in M$, we define \emph{the growth vector} $\underline{n}(x) = (n_j(x))_{j=1}^\infty$ of $D$ at $x$ by
$$n_j(x) = \rank \left\{ X(x) \, : \, X \in \underline{D}^j \right\}.$$
We say that $\lambda$ is \emph{the step} of $D$ at $x$ if
$$\lambda = \min \left\{ k \in \mathbb{N} \, : \, \text{$n_j(x) = n_k(x)$ for any $j \geq k$} \right\}.$$
We often give the growth vector at $x$ only by $\underline{n}(x) = (n_1(x), \dots, n_\lambda(x))$, indicating that the remaining numbers equal $n_\lambda(x)$.

We say that $D$ is \emph{bracket-generating} if for every $x$ there exists some $j$ such that $n_j(x) = \dim M$. In other words, $D$ is called bracket-generating if the sections of~$D$ and their iterated brackets span the entire tangent bundle. If $D$ is bracket-generating, then the distance $\sfd$ defined in \eqref{CC} is always finite and its topology coincides with that of the manifold.

If $\underline{n}(\blank)$ is constant in a neighborhood of $x$, then $x$ is called a \emph{regular point} of $D$. The point $x$ is called \emph{singular} if it is not regular. The set of singular points of $D$ is closed with empty interior, cf.~\cite[Sect.\ 2.1.2, p.\ 21]{Jea14}. If all points of $M$ are regular, $D$ is called \emph{equiregular}. If $D$ is equiregular, the step $\lambda$ is constant and there is a flag of subbundles of $TM$,
$$D^1 = D \subseteq D^2 \subseteq \cdots \subseteq D^\lambda,$$
such that $\underline{D}^j = \Gamma(D^j)$.

\begin{remark}
A sub-Riemannian metric can alternatively be described as a possibly degenerate cometric $g^*$ on $T^*M$ such that the image $D$ of the map $\sharp: T^*M \to TM$, $\beta \mapsto \langle \beta, \blank \rangle_{g^*}$ is a subbundle. These two points of views are related by
$$\langle \beta, \alpha \rangle_{g^*} = \langle \sharp \beta, \sharp \alpha \rangle_{g}.$$
\end{remark}

\begin{remark}
We could have considered an even more general setting where we have a Finsler metric rather than a Riemannian metric on $D$, thus considering sub-Finsler geometry. The techniques and results of this paper are still valid in this sub-Finsler setting.
\end{remark}

\subsection{Carnot groups} \label{sec:Carnot}
Let $\frakn$ be a nilpotent Lie algebra. Assume that this Lie algebra is stratified, meaning that we can write $\frakn= \frakn_1 \times \cdots \times \frakn_\lambda$ with bracket relations satisfying
$$[\frakn_1, \frakn_j] = \left\{ \begin{array}{ll} \frakn_{j+1} & \text{if $1 \leq j \leq \lambda-1$,} \\ 0 & \text{if $j=\lambda$.} \end{array} \right. $$
Let $N$ be the connected, simply connected Lie group corresponding to $\mathfrak{n}$. Choose an inner product on $\mathfrak{n}_1$ and define a sub-Riemannian structure $(D,g)$ on $N$ by left translation of $\mathfrak{n}_1$ and this inner product. The sub-Riemannian manifold $(N,D,g)$ is then called a Carnot group of step $\lambda$. Observe that the case $\lambda =1$ is just an inner product space.

With slight abuse of notation, we will use $0$ for the identity element of $N$, even though~$N$ is abelian only if $\lambda =1$. For each $s >0$, define \emph{the dilation}~$\delta_s$ as the Lie group automorphism $\delta_s\colon N \to N$ uniquely determined by
\begin{equation} \label{CarnotDilations} (\delta_{s})_* A \in \frakn^j \mapsto s^j A.\end{equation}
Dilation $\delta_s$ is homothetic: it multiplies the Carnot-Carath\'eodory distance by $s$. Therefore, as a metric space, $(N,D,g)$ is geodesic, locally compact, it has a transitive group of isometries and homotheties. Conversely, a theorem of Le Donne \cite{LD15} asserts that metric spaces with such properties are Finsler Carnot groups.

Write $n_0= 0$ and define  $n_j$ as the dimension of $\frakn_1 \times \cdots \times \frakn_j$ for any $1 \leq j \leq \lambda$. Choose a basis $Z_1, \dots, Z_n$ of $\frakn$ such that
\begin{equation} \label{WeightedBasis} \spn\{ Z_{n_{j-1} + 1} , \dots, Z_{n_j} \} = \mathfrak{n}_j.
\end{equation}
Since $N$ is simply connected and nilpotent, the exponential map is a global diffeomorphism. We introduce exponential coordinates $(z_1, \dots, z_n)$ corresponding to the point $\exp (z_1 Z_1 + \dots + z_n Z_n)$. In these coordinates, we have
$$\delta_s\colon (z_1, \dots, z_i, \dots, z_n) \mapsto (s z_1, \dots, s^{w_i} z_i, \dots, s^\lambda z_n),$$
where $w_i$ is the number such that $Z_i \in \frakn_{w_i}$. We call $w_i$ \emph{the $i$-th weight}.

We define \emph{the radial vector field} $S$ by $S(x) =  \left. \frac{d}{dt} \delta_{1+t}(x) \right|_{t=0}.$
It follows that if $t \mapsto e^{tS}$ is the flow of $S$, then $\delta_s = e^{(\log s) S}$ for any $s > 0$. If $Z_1, \dots, Z_n$ is a basis of $\mathfrak{n}$ satisfying \eqref{WeightedBasis} and $(z_1, \dots, z_n)$ are the corresponding exponential coordinates, then
$$S = \sum_{i=1}^n w_i z_i \frac{\partial}{\partial z_i}.$$
Furthermore, since $\delta_s$ is a group homomorphism, it commutes with the left action, meaning that we also have
\begin{equation} \label{sfRLeft} S = \sum_{i=1}^n w_i z_i Z_i,\end{equation}
Here, we have used the same symbols for elements in $\mathfrak{n}$ and their corresponding left invariant vector fields.

As all Carnot groups are contractible spaces, we know that any principal bundle $\pi:P \to M$ can be trivialized. Furthermore, we have the following stronger statement by using our dilations $\delta_s$. Let $\pi:P \to N$ be a principal bundle with a connection $\omega$. Let $hS$ denote the horizontal lift of $S$ with respect to $\omega$ and let $t \mapsto e^{thS}$ be its flow. Define $\delta_s^\omega: P \to P$ by formula $\delta_s^\omega (p) = e^{(\log s)  hS}(p)$.
\begin{lemma} \label{lemma:RadialG}
\begin{enumerate}[\rm (a)]
\item There exists a unique principal bundle isomorphism $\zeta: P \to N \times P_0$ such that
\begin{enumerate}[\rm (i)]
\item $\zeta(p_0) = (0, p_0)$ for any $p_0 \in P_0$,
\item if $\zeta(p) = (x, p_0)$, then $\zeta( \delta_s^\omega(p)) = (\delta_s(x), p_0)$.
\end{enumerate}
\item For every $p_0 \in P_0$, there is a unique gauge $\sigma: N \to P$ such that $\sigma(0) = p_0$ and $(\sigma^*\omega)(S) = 0$. Furthermore, this satisfies $(\sigma^* \omega)(v) = 0$ for every $v \in T_0 N$.
\end{enumerate}
\end{lemma}

\begin{proof}
\begin{enumerate}[\rm (a)]
\item For a fixed $x$, define $\gamma(t) = \delta_{1-t}(x)$. Let $\gamma_p(t)$ denote this curve's horizontal lift to $p \in P_x$. Define then
$$\zeta(p) = (x, \gamma_p(1)).$$
By smooth dependence on initial conditions of ODEs, this is a well defined smooth map. Furthermore, from properties of horizontal lifts, $\zeta(p \cdot a) = \zeta(p) \cdot a$. Hence, $\zeta$ is a principal bundle isomorphism.
\item Define $\sigma$ by the relation $\zeta(\sigma(x)) = (x, p_0)$. We then have $\sigma_* S = hS\circ\sigma$ which is obviously in the kernel of $\omega$. The last statement follows from the fact that there exist points $x_1,\ldots,x_n\in N$ such that tangent vectors at $0 \in N$ of the form $\frac{d}{ds} \delta_{s}(x) |_{s=0}$ form of basis of $T_0 N$.
\end{enumerate}
\end{proof}
We call any gauge such as in Lemma~\ref{lemma:RadialG}~(b) \emph{radial, centered at $0$}. We can define radial gauges centered at an arbitrary point $x \in N$ similarly, replacing dilations $\delta_s$ with $\delta^x_s := l_x \circ \delta_s \circ l_{x^{-1}}$, with $l_x$ denoting left translations by $x$.

\subsection{Weighting of functions, forms and vector fields} \label{sec:WeightC}
For any $q$-form $\alpha$ on $N$ with $q \geq 0$, we define $\delta_s^* \alpha$ as the pull-back with respect to $\delta_s$. 
We say that a $q$-form $\alpha$ is \emph{homogeneous of weight} $m \in \mathbb{Z}$ if $\delta_s^* \alpha = s^m \alpha$. We say that $\alpha$ is of \emph{weight} $\geq m$ if $\lim_{s\downarrow 0} s^{-k} \delta_s^* \alpha = 0$ for any $k < m$. We remark that if $\alpha_j$ is homogeneous of weight $m_j$ (resp. of weight $\geq m_j$) for $j =1, 2$, then $\alpha_1 \wedge \alpha_2$ is homogeneous of weight $m_1 + m_2$ (resp. of weight $\geq m_1 + m_2$) and $\alpha_1 + \alpha_2$ is of weight $\geq \min\{m_1,m_2\}$.

Write $\sfd_0(x) := \sfd(0,x)$. Since $\{ \delta_s \}_{s > 0}$ is a dilation, any function $f$  of weight $\geq m$ satisfies $f = O(\sfd^m_0)$ as $x \to 0$. For a general smooth function $f$, we have a \emph{weighted Taylor polynomial at $0$}, $\Tay_0(f; k) = \sum_{m=0}^k f^{(m)}$, with each $f^{(m)}$ being homogeneous of weight $m$. Each $f^{(m)}$ is defined by
\begin{equation}
\label{mComponent} f^{(m)} = \left. \frac{1}{m !} \frac{d^m}{d s^m} \delta_s^* f \right|_{s=0}.\end{equation}
We can also give a similar Taylor expansion for forms. The one-forms $dz_1, \dots, dz_n$ are a global basis of $T^*N$ and $dz_i$ is homogeneous of weight $w_i$. It follows that every one-form $\alpha$ is of weight $\geq 1$ and has a weighted Taylor expansion~$\Tay_0(\alpha;k) = \sum_{m=1}^k \alpha^{(m)}$ as well. Here
$$\alpha^{(m)} = \sum_{i=1}^n f_i^{(m- w_i)} dz_i, \qquad \text{whenever } \alpha = \sum_{i=1}^n f_i dz_i,$$
with the convention that $f^{(m)} = 0$ whenever $m < 0$. Similarly, any $q$-form $\alpha$ has weight $\geq q$ and Taylor expansion $\Tay_0(\alpha; k)$ by defining $(\alpha \wedge \beta)^{(m)} = \sum_{k=1}^m \alpha^{(k)} \wedge \alpha^{(m-k)}$.

For a vector field $X$ on $N$, we define $\delta^*_sX$ such that for any function $f$,
$$(\delta_s^* X)(\delta_s^* f) = \delta_s^* (Xf).$$
In other words,
$$\delta_s^* X(x) = (\delta_{1/s})_* X(\delta_s(x)).$$
We define homogeneous vector fields and vector fields of weight $\geq m$ in analogy with our definition on forms.
If $X$ is any vector field of weight $\geq m$ and $f$ is any function, then $fX$ is of weight $\geq m$ as well. Furthermore, since
\begin{equation} \label{BracketDelta} [\delta_s^* X_1, \delta_s^* X_2] = \delta_s^* [X_1,X_2], \qquad X_1,X_2 \in \Gamma(TM).\end{equation}
it follows that if $X_j$ is homogeneous of weight $m_j$ (resp. of weight $\geq m_j$), then $[X_1, X_2]$ is homogeneous of weight $m_1 + m_2$ (resp. of weight $\geq m_1 + m_2$). Consider $Z_1, \dots, Z_n$ as the basis of $\mathfrak{n}$ satisfying \eqref{WeightedBasis}. Use the same symbols for the corresponding left invariant vector fields. By definition, $Z_i$ is a homogeneous vector field of weight $-w_i$ for $1 \leq i \leq n$. As a consequence, any vector field with values in $D$ has weight $\geq -1$ and all vector fields are of weight $\geq -\lambda$.

Observe that any homogeneous function, vector field or form is uniquely determined by their values in a neighborhood $U$ of $0$. Hence, we call a function on such a neighborhood $U$ \emph{homogeneous} if it is the restriction of a homogeneous function on~$N$. Similarly, for any neighborhood $U$ of $0$, there exists a neighborhood $0 \in \tilde U \subseteq U$ such that $\delta_s(\tilde U) \subseteq \tilde U$ for $s \leq 1$. We say that a function $f$ defined on $U$ is of \emph{weight $\geq m$} if $\lim_{s \downarrow 0} s^k \delta_s^* f|_{\tilde U} =0$ for any $k < m$. We use analogous definitions for forms and vector fields.

\begin{remark}
Notice that since the flow of $S$ commutes with $\delta_s$, the vector field $S$ is homogeneous of weight $0$.
\end{remark}

\subsection{Approximation of holonomy}
Let $M$ be a manifold. Define $\scrL(x, M)$ as the space of all absolutely continuous loops $\gamma:[0,1] \to M$ based at $x \in M$ with finite length relative to some (and hence any) complete Riemannian metric $\tilde g$ on $M$. 

\begin{definition} \label{def:Ol}
Let $\sfd$ be a metric on $M$. Let $\ell = \ell^{\sfd}: \scrL(x, M) \to [0, \infty]$ be the length relative to $\sfd$,
$$\ell(\gamma) = \sup \left\{\sum_{j=1}^k \sfd\big(\gamma(t_{j-1}), \gamma(t_j) \big) \, : k \geq 1, \quad 0 = t_0 < t_1 < \cdots < t_k =1 \right\} .$$
 Let $f$ be a real-valued function defined on a subset of $\mathbb{R}$ and let $F$ be a function defined on a subset of $\scrL(x, M)$ with values in a normed vector space $(E, \| \blank\|)$. We then say that $F = O(f(\ell))$ as $\ell \to 0$ if there exist positive constants~$\eps$ and~$C$ such that
 \begin{enumerate}[\rm (i)]
\item $f$ is defined on $(-\eps, \eps)$,
\item $F(\gamma)$ is well defined for any $\gamma$ with $\ell(\gamma) < \eps$,
\item for any $\ell(\gamma) < \eps$, we have
$$\| F(\gamma)\| \leq C f(\ell(\gamma)).$$
\end{enumerate}
\end{definition}
Note that the definition of $O(f(\ell))$ does not change if we replace the norm on~$E$ with an equivalent norm. Hence, if~$E$ is finite dimensional, we do not need to specify a norm in Definition~\ref{def:Ol}, since all norms are equivalent.

Let us now consider the special case when $(M, D,g)$ is a sub-Riemannian manifold and $\sfd$ is defined as in \eqref{CC}. The following observations are important to note.
\begin{enumerate}[\rm (i)]
\item If $\gamma$ is any absolutely continuous curve such that $\ell(\gamma) < \infty$, then $\gamma$ must be horizontal.
\item Any absolutely continuous curve $\gamma: [0,1] \to M$ in a Riemannian manifold $(M, \tilde g)$ can be parametrized by arc length, meaning that we can find a reparametrization $\tilde{\gamma}:[0,1] \to M$ satisfying $| \dot {\tilde{\gamma}}(t) |_{\tilde g} = \ell(\gamma)$ whenever $\dot{ \check{\gamma}}(t)$ exists, see e.g. \cite[Chapter 5.3]{Pet06}. By a similar argument, we can reparametrize any horizontal curve in a sub-Riemannian manifold $(M,D, g)$ to have constant speed.
\end{enumerate}

Let $\pi: P \to N$ be a principal bundle with connection $\omega$ with curvature $\Omega$. For any $\gamma \in \mathscr{L}(x, M)$, we define the corresponding \emph{radial disk} $\disk(\gamma): [0,1]^2 \to M$ by
$$\disk( \gamma)(s,t) = \delta_s^x(\gamma(t)).$$
\begin{theorem} \label{th:Flat} 
Let $p \in P_x$, $x \in N$ be an arbitrary point. Let $\sigma$ be a radial gauge centered at $x$ and satisfying $\sigma(x) = p$. Define $F_p: \scrL(x, M) \to \mathfrak{g}$ by
$$F_p(\gamma) = \int_{\disk(\gamma)} \sigma^*\Omega.$$
Assume that $\sigma^* \Omega$ is of weight $\geq m_\Omega$. Then
$$\Hol_p^\omega = \exp\left( - F_p+ O(\ell^{2m_\Omega}) \right).$$
\end{theorem}
The proof is given in Section~\ref{sec:FlatProof}

\subsection{Theory and proof for Theorem~\ref{th:Flat}}
\subsubsection{Metric spaces with dilations}
Let $(M, \sfd)$ be a metric space with possibly some points being of infinite distance. We call a collection of maps $\{\delta_s: M \to M \}_{s > 0}$ \emph{dilations} if
$$\delta_1 = \id_M, \qquad \delta_{s_1} \circ \delta_{s_2} = \delta_{s_1 s_2}.$$
and if $\sfd(\delta_s(x), \delta_s(y)) = s \, \sfd(x,y)$.

Assume that $M$ is a smooth manifold and define $\scrL(M) = \coprod_{x \in M} \scrL(x, M)$. For any curve $\gamma \in \scrL(M)$, define $\delta_s \gamma$ as the loop $\delta_s\gamma(t) = \delta_s(\gamma(t))$. Let $\sfd$ be a metric on~$M$ and assume that this metric has dilations $\{ \delta_s\}_{s >0}$ that are also diffeomorphisms. 
\begin{lemma} \label{lemma:StoL}
 Let $F: \scrL(M) \to E$ be a function with values in a normed vector space $(E, \| \blank \|)$. Let $x_0 \in M$ be an arbitrary point. Assume that there are constants $\eps \in(0,1]$ and $C > 0$ such that for every $y \in \{ \delta_s (x_0)\}_{s \geq 1}$ and $\gamma \in \scrL(y, M)$ with $\ell(\gamma) \leq \eps$,
$$\|F(\delta_s \gamma)\| \leq C f(\eps s), \qquad \text{whenever $0 < s \leq 1$}.$$
Then $F = O(f(\ell))$ at $x_0$.
\end{lemma}
\begin{proof}
For any $\gamma \in \scrL(x_0)$ with $\ell(\gamma) < \eps$, define $\tilde \gamma = \delta_{\eps/ \ell(\gamma)} \gamma$. Then by our assumptions
$$\| F(\gamma) \| = \| F(\delta_{\ell(\gamma)/\eps} \tilde \gamma) \| \leq C f(\ell(\gamma)).$$
\end{proof}

Notice that $F$ in Lemma~\ref{lemma:StoL} needs only be defined for any $\gamma \in \scrL(y, M)$, $y \in \{ \delta_s (x_0)\}_{s \geq 1}$ with $\ell(\gamma) \leq \eps$. If $x_0$ is a fixed point for every $\delta_s$, then we only need to consider $y = x_0$ in Lemma~\ref{lemma:StoL}. If $\sfd$ only has finite values and $\delta_s$ has a fixed point~$x_0$, then this point is unique since $\lim_{s \downarrow 0} \delta_{s}(x) = x_0$ for any $x\in M$.

For a Carnot group, the maps $\delta_s$ defined in \eqref{CarnotDilations} are dilations in the above sense with fixed point $0$. In general, for every $x \in N$, the maps $\delta_s^x = l_x \circ \delta_s \circ l_{x^{-1}}$ are dilations with fixed point $x$.

\subsubsection{Forms on Carnot groups and functionals}
Let $(N, D, g)$ be a Carnot group with dilations $\delta_s$. We want to consider functions on loop space $\scrL(x,N)$ related to forms.
\begin{lemma} \label{lemma:PhiAlpha}
Let $\alpha$ be a two-form on $N$ with values in some finite dimensional vector space $\mathfrak{g}$. Assume that $\alpha$ is of weight $\geq m$ at $x$.
Consider the map
\begin{equation} \label{PhiAlpha} \Phi^\alpha: \scrL(x,N) \to \mathfrak{g}, \qquad \Phi^\alpha(\gamma) = \int_{\disk(\gamma)} \alpha.\end{equation}
Then $\Phi^\alpha = O(\ell^m)$.
\end{lemma}
\begin{proof}
It is sufficient to consider $\mathfrak{g} = \mathbb{R}$ and $x = 0$. Choose an arbitrary loop $\gamma \in \mathcal{L}(0,U)$ and assume that it is parametrized to have constant speed, i.e.  $| \dot \gamma(t) |_g = \ell(\gamma)$ at all points where the derivative exists. Write $\iota_S \alpha = \sum_{j=1}^n f_j dz_j$ and define
\begin{equation} \label{BetaS} \beta_s = \sum_{j=1}^n s^{w_j} \left(\int_0^1 r^{w_j-1} (\delta_{sr}^*f_j) dr \right) dz_j.\end{equation}
We compute
$$ \int_{\disk(\delta_s \gamma)}\alpha = \int_0^1 \int_0^1 \frac{1}{r} (\delta_{sr}^*\alpha)(S, \dot \gamma(t)) \, dr \, dt = \int_0^1 \beta_s(\dot \gamma(t)) \, dt,$$
and so
$$\left\| \Phi^\alpha (\gamma) \right\| \leq  \ell(\gamma) \int_0^1 | \beta_s |_{g^*}(\gamma(t)) \, dt.$$
Since $\iota_S \alpha$ is of weight $ \geq m$, each $f_j$ is of weight $\geq m-w_j$. It follows that for sufficiently small $\eps > 0$, there are constants $C_j$, $1 \leq j \leq n_1$ such that $\|f_j(x) \| \leq C_j \sfd(x)^{m-1}$ whenever $\sfd_0( x) < \eps$. Combining these bounds, we obtain,
$$|\beta_s|_{g^*}(x) \leq \frac{s^{m}}{m} \sum_{i=1}^{n_1} C_j \sfd_0(x)^{m-1} .$$
Defining $C_0 = \max_{1 \leq j \leq n_1} C_j$ and using that $\sfd_0(\gamma(t)) \leq \frac{\ell(\gamma)}{2}$, gives us the inequality
$$\|\Phi^\alpha(\gamma)\| \leq \frac{n_1 \ell(\gamma)^{m} C_0}{2^m m } s^{m} <  \frac{n_1 \eps^{m} C_0}{m } s^{m} ,$$
whenever $\ell(\gamma) < 2\eps$. The result follows from Lemma~\ref{lemma:StoL}.
\end{proof}

\subsubsection{Proof of Theorem~\ref{th:Flat}} \label{sec:FlatProof}

Choose an embedding of $\mathfrak{g}$ into $\gl(q, \mathbb{C})$ and put a Banach algebra norm $\| \blank \|$ on $\mathfrak{gl}(q, \mathbb{C})$. We will show that there exists an $\eps >0$ and $C$ such that
\begin{enumerate}[\rm (i)]
\item $\exp^{-1} \Hol^\omega_{\sigma(0)}(\gamma)$ is well defined whenever $\ell(\gamma) < \eps$,
\item For every $0 < s \leq 1$ and $\gamma$ with $\ell(\gamma) < \eps$, 
$$\|\exp^{-1} \Hol(\delta_s\gamma) + \Phi^{\sigma^* \Omega}(\delta_s\gamma) \| \leq C s^{2m_\Omega}.$$
\end{enumerate}

Let $\gamma$ be an arbitrary curve with length $\ell(\gamma) < \eps$. Without loss of generality, we may assume that $\gamma$ is parametrized to have constant speed, so $| \dot \gamma(t) |_g = \ell(\gamma)$ whenever it is defined. Let $\sigma$ be a radial gauge. Give $p \in P$ coordinates $(x, a) \in N \times G$ if $p = \sigma(x) \cdot a$ and define $\gamma_s = \delta_s\gamma$. In these coordinates, the horizontal lift $(\gamma_s(t), a_s(t))$ of $\gamma_s(t)$ to $\sigma(0)$ is given by
$$\dot a_s(t) \cdot a_s(t)^{-1} = - \sigma^* \omega(\dot \gamma_s(t)) =: A_s(t), \qquad a_s(0) = 1,$$
and $a_s(1) = \Hol^\omega_{\sigma(0)}(\gamma_s).$

By Lemma~\ref{lemma:Picard}, there is an $\eps_0 < 1$ such that whenever
$$\| A_s\|_{L^\infty} = \max_{t \in [0,1]} \|A_s(t) \|_{L^\infty} < \eps_0,$$
then $\exp^{-1} a_s(t)$ is well defined and
$$\left\| \exp^{-1} a_s(1) - \int_0^1 A_s(t) \, dt \right\|   \leq C_2 \| A_s\|^2_{L^\infty},$$
for some constant $C_2$. Since $(\sigma^* \omega)(S) = 0$,
\begin{align*}
- A_s(t) & = \int_0^s \frac{\partial}{\partial r} (\sigma^* \omega)(\dot \gamma_r) \, dr = \int_0^s d(\sigma^* \omega)\left( \frac{\partial}{\partial r} \gamma_r, \dot \gamma_r \right) \, dr \\
&  = \int_0^s \frac{1}{r} (\sigma^* \Omega)\left( S(\gamma_r), \dot \gamma_r \right) \, dr  .
\end{align*}
Since $\iota_S \sigma^*\Omega$ is of weight $\geq m_\Omega$, there is a constant $C_3$ such that
$$\| A_s\|_{L^\infty} \leq C_3 s^{m_0} \ell(\gamma)^{m_0}.$$
Hence, if we choose $C_3 \eps^{m_0} < \eps_0$,
$$\left\| \exp^{-1} a_s(1) - \int_0^1 A_s(t) \, dt \right\| = \left\| \exp^{-1} \Hol^\omega_p (\delta_s\gamma) - F_p(\delta_s \gamma) \right\| \leq  C_2C_3^2 \eps^{2m_0} s^{2m_\Omega}.$$
The result now follows.

\section{Gauge-free formulation} \label{sec:GaugeFree}
\subsection{Equivariant forms and adjoint bundle}
Let $\pi: P \to M$ be a principal bundle with structure group $G$. Let $\mathfrak{g}$ be the Lie algebra of $G$. We can then introduce a vector bundle $\Ad(P) \to M$ called \emph{the adjoint bundle} by considering the product bundle of $P$ and $M \times \mathfrak{g}$ divided out by the right action
$$(p, A) \cdot a = (p \cdot a, \Ad(a^{-1}) A), \qquad p \in P, A \in \mathfrak{g}, a \in G.$$
We write $[p,A] \in \Ad(P)$ for the equivalence class of $(p,A) \in P \times \mathfrak{g}$.

Sections of $\Ad(P)$ correspond to \emph{$G$-equivariant} functions on $P$. A function $\varphi: P \to \mathfrak{g}$ is called $G$-equivariant if $f(p \cdot a) = \Ad(a^{-1}) f(p)$. This function can be considered as a section of $\Ad(P)$ through the identification $\varphi(x) = [p, \varphi(p)]$, $p \in P_x$. Similarly, a $\mathfrak{g}$-valued form $\alpha$ on $P$ is called $G$-equivariant if
$$\alpha(v_1 \cdot a, \dots, v_q \cdot a) = \Ad(a^{-1}) \alpha(v_1, \dots, v_q).$$
Any $G$-equivariant form that vanishes on $\ker \pi_*$ can be considered as a form on $M$ with values in $\Ad(P)$.

Let $\omega$ be a connection on $\pi$. Let $hX$ denote the horizontal lift of a vector field $X$ with respect to $\omega$. For any $G$-equivariant function $\varphi$, the function $hX \varphi$ is still $G$-equivariant. We denote the corresponding section of $\Ad(P)$ by $\nabla_X^\omega \varphi$, giving us an affine connection $\nabla^\omega$ on $\Ad(P)$. Remark that for any gauge $(\sigma, U)$, we have
\begin{equation} \label{SigmaDer} \sigma^*\nabla_X^\omega \varphi = X \sigma^*(\varphi)+ [\sigma^* \omega(X), \sigma^*\varphi].\end{equation}

\subsection{Explicit formulas of weighted components} \label{sec:Explicit}
Let $N$ be a Carnot group with stratified Lie algebra $\mathfrak{n}$ and with exponential coordinates $z$. We give two formulas for the $m$-th homogeneous component of a general function $f$ relative to $0 \in N$. Let $n$ be the dimension of $N$ and define an index set $J_n$ by
\begin{equation} \label{Jindex} J_n(0) = \{ \emptyset \}, \qquad J_n(j) = \{ 1, \dots, n\}^j, \qquad J_n = \bigcup_{j=0}^\infty J_n(j).
\end{equation}
and define functions $j, w, |\blank|: J_n \to \mathbb{N}$ by
\begin{equation} \label{SizeIndex}
 j(\mu) = j, \qquad  w(\mu) = \sum_{i=1}^j w_{\mu_i} , \qquad |\mu| =\sum_{i=1}^j \mu_i, \qquad \text{for any $\mu \in J_n(j)$.}
\end{equation}
We use the convention that $w(\emptyset) = 0$ and $|\emptyset| = 0$.

By rearranging the terms of the usual Taylor expansion, we have
\begin{equation} \label{TaylorExp} f^{(m)}(z) = \sum_{\begin{subarray}{c} \mu \in J_n \\ w(\mu) = m \end{subarray}} \frac{1}{j(\mu)!} \frac{\partial f}{\partial z_\mu}  (x_0) z_\mu,
\end{equation}
where $z_\mu = z_{\mu_1}  \cdots z_{\mu_j}$ and  $\frac{\partial}{\partial z_\mu} = \frac{\partial}{\partial z_{\mu_1}} \cdots \frac{\partial}{\partial z_{\mu_j}}$ for any $\mu \in J_n(j)$.
We use the conventions $z_\emptyset = 1$ and $\frac{\partial}{\partial z_\emptyset}f = f$. On the other hand, for any element $Z \in \mathfrak{n}$, we have the Taylor expansion $f( \exp t Z) = \sum_{j=0}^k \frac{1}{j!} t^j(Z^jf)(0) + O(t^{k+1}).$
Since $z_1, \dots, z_n$ are exponential coordinates, we obtain
$$f(z_1, \dots z_n) = \sum_{\begin{subarray}{c} \mu \in J_n \\ j(\mu) \leq k \end{subarray}} \frac{z_\mu}{j(\alpha)!} (Z_\mu f)(0) + O(|z|^{k+1}).$$
with $Z_\mu$ defined analogously to $\frac{\partial}{\partial z_\mu}$. Collecting all the homogeneous components of this sum with $k$ sufficiently large, we obtain
\begin{equation} \label{HomLeft} f^{(m)}(z) = \sum_{\begin{subarray}{c} \mu \in J_n \\ w(\mu) = m \end{subarray}} \frac{z_\mu}{j(\mu)!} (Z_\mu f)(0). \end{equation}

\subsection{Weighted components of equivariant forms} \label{sec:WEF}
In order to introduce approximations that are independent of choice of gauge, we will introduce Taylor expansions of $\Ad(P)$-valued forms. Let $\pi: P \to N$ be a principal bundle with connection~$\omega$. Define $\delta_s^\omega : P \to P$ and $\zeta: P \to N \times P_0$ as in Section~\ref{sec:Carnot}. For any $\varphi \in \Gamma(\Ad(P))$ and $p \in P_0$, introduce $\Tay_p^\omega(\varphi;k) = \sum_{m=0}^k\varphi^{(m)}$ where $\varphi^{(m)}: N \to \mathfrak{g}$ is defined by
$$\varphi^{(m)}(x) = \left. \frac{d^m}{ds^m} \varphi\left( \delta_s^\omega(\zeta^{-1}(x, p)) \right) \right|_{s=0}.$$
We emphasize that $\Tay_p^\omega(\varphi; k)$ is a function on $N$, not on $P$.

We want to use \eqref{TaylorExp} to give an explicit representation of $\varphi^{(m)}$. Write $\partial_j = \frac{\partial}{\partial z_j}$ and let $\nabla$ be the connection on $N$ determined by $\nabla \partial_j = 0$. Denote the induced connection by $\nabla$ and $\nabla^\omega$ on $\Ad(P)$-valued forms by the symbol $\nabla^{\omega}$ as well. Next, for any $j \geq 1$, define $\nabla^{\omega, 1}_X \varphi = \nabla^{\omega}_X \varphi$ and iteratively
$$\nabla^{\omega,j+1}_{X_0, X_1, \dots, X_j} \varphi = \nabla_{X_0}^\omega \nabla^{\omega,j}_{X_1, \dots, X_j} \varphi -  \nabla^{\omega,j}_{\nabla_{X_0} X_1, \dots, X_j} \varphi - \cdots - \nabla^{\omega,j}_{X_1, \dots, \nabla_{X_0} X_j} \varphi. $$
Finally, for any $\mu \in J_n(j)$, we write
$$\nabla_\mu^\omega = \nabla^{\omega,j}_{\partial_{\mu_1}, \dots, \partial_{\mu_j}}.$$
Then we have the expansion
\begin{equation} \label{HomEquivariant} \varphi^{(m)}(z) = \sum_{\begin{subarray}{c} \mu \in J_n \\ w(\mu) = m  \end{subarray}} \frac{z_\mu}{j(\mu)!} (\nabla_\mu^{\omega} \varphi)(p). \end{equation}

Alternatively, we can use \eqref{HomLeft}. Define a connection $\bnabla$ by the rule $\bnabla Z_j = 0$. Define $\bnabla^{\omega, j}$ analogously to the definition of $\nabla^{\omega,j}$ and furthermore
$$\bnabla^{\omega}_\mu = \bnabla^{\omega,j}_{Z_{\mu_1}, \dots, Z_{\mu_j}},$$
giving us
$$\varphi^{(m)}(z) = \sum_{\begin{subarray}{c} \mu \in J_n \\ w(\mu) = m \end{subarray}} \frac{z_\mu}{j(\mu)!} (\bnabla_\mu^{\omega} \varphi)(p)$$
as well.

We extend the definition of $\Tay_p^\omega(\alpha; k) = \sum_{m=q}^k \alpha^{(m)}$ to a general $q$-form $\alpha$ with values in $\Ad(P)$ as in Section \ref{sec:WeightC}. Remark that
$$\Tay_{p\cdot a}^\omega(\alpha; k) = \Ad(a^{-1}) \Tay_p^\omega(\alpha;k),$$
and, in particular, the number $0 \leq m_\alpha \leq \infty$ defined by
\begin{equation} \label{mAlpha} m_\alpha = \sup \left\{ m \in \mathbb{Z} \, : \, \alpha^{(m)}(0) = 0\right\},\end{equation}
is independent of $p \in P_0$.

\subsection{Gauge-free Taylor expansion}
The curvature for $\Omega$ is an $\Ad(P)$-valued valued $2$-form on $M$. We will define an expansion of holonomy using the Taylor expansion of $\Omega$ that does not depend on gauge.
\begin{theorem} \label{th:Taylor}
Let $\pi: P \to N$ be a principal bundle with connection $\omega$. Let $p \in P_0$ be an arbitrary element. Define
$$F^k_p(\gamma) = \int_{\disk(\gamma)} \Tay_p^\omega(\Omega; k).$$
Then for any $k< 2m_\Omega$,
$$\Hol_p^\omega = \exp \left( - F_p^k + O(\ell^{k+1}) \right).$$
\end{theorem}

\begin{remark}
$F_p^k(\gamma)$ is a linear combination of terms involving iterated covariant derivatives of $\Omega$ at $0$, whose coefficients are moments (integrals of polynomial 1-forms along $\gamma$). It is gauge free in the sense that $\nabla^\omega_\mu \Omega(p)$ does not depend on a choice of local gauge. This will be more apparent in the computation of examples in section \ref{sec:Examples}.
\end{remark}

\subsubsection{Proof of Theorem~\ref{th:Taylor}} \label{sec:TaylorProof}
By Theorem~\ref{th:Flat}
$$\Hol_p^\omega = \exp \left( - \Phi^{\sigma^* \Omega} + O(\ell^{2 m_0}) \right).$$
We write $\Phi^{\sigma^* \Omega} = F^k_p + \Phi^\alpha$ such that
$$\alpha = \sigma^* \Omega - \Tay_p( \sigma^* \Omega; k).$$
Since $\alpha$ is of order $\geq k+1$, the result follows from Lemma~\ref{lemma:PhiAlpha}.

\section{Privileged coordinates and general Riemannian and sub-Riemannian spaces} \label{sec:Curved}
\subsection{Privileged coordinates}
Let us now consider a sub-Riemannian manifold $(M,D,g)$ with $D$ being bracket-generating, but with no further assumptions.
Such manifolds do not have dilations in general. However, we can locally find coordinate systems that play a role similar to the exponential coordinates of Section~\ref{sec:Carnot}.

Let $x_0$ be a regular point of $M$. Let $U$ be a neighborhood of $x_0$ where the growth vector $\underline{n}(\blank)$ is constant. All points in this neighborhood will be regular and of the same step $\lambda$. Hence, by replacing $M$ with $U$, we may assume that the horizontal bundle $D$ is equiregular with constant growth vector $\underline{n} = (n_1, n_2, \dots )$. Relative to this growth vector, for any $1 \leq i \leq n$, define \emph{the $i$-th weight} $w_i$ as the integer uniquely determined by
$$n_{w_i-1} < i \leq n_{w_i},$$
with the convention that $n_0 = 0$.

\begin{definition}
Let $D$ be a bracket-generating, equiregular subbundle with corresponding flag $D \subseteq D^2 \subseteq \cdots \subseteq D^\lambda = TM$ and growth vector $\underline{n} =(n_1, n_2, \dots, n_\lambda)$. Let $z: U \subseteq M \to \mathbb{R}^n$ be a local coordinate system such that $z(U) = \mathbb{R}^n$ and $z(x_0) =0$ for some $x_0 \in U$.
\begin{enumerate}[\rm (a)]
\item The coordinate system $(z,U)$ is called linearly adapted at $x_0$ if $dz_i( D^j_{x_0}) =0$ whenever $n_j < i$.
\item A function $f$ defined in some neighborhood of $x_0$ is said to have order~$\geq m$ at~$x_0$ if for any $j =0, 1,\dots, m-1$ and any collection of vector fields $X_1, \dots, X_j \in \Gamma(D)$, we have
$$(X_1 \cdots X_j f)(x_0) = 0.$$
Furthermore, $f$ is said to have order $m$ if it is of order $\geq m$ but not of order $\geq m+1$.
\item The coordinate system $(z,U)$ is called privileged if $z_i$ has order $w_i$ at $x_0$. Any such coordinate system will be linearly adapted at $x_0$.
\end{enumerate}
\end{definition}
For any sub-Riemannian manifold $(M, D, g)$ and any $x_0\in M$, there exist privileged coordinate systems centered at $x_0$. For a construction, see \cite[Section 4.3]{Bel96}. Note that the definition of a privileged coordinate system at $x_0$ for a sub-Riemannian manifold $(M,D,g)$ only depends on $D$, but not on $g$. The following characterization of the order of functions is found in \cite[Proposition~4.10]{Bel96}.
\begin{proposition}
Let $f$ be a function defined around $x_0$. Let $\sfd$ be the Carnot-Carath\'eodory metric of $(M,D,g)$. The following are equivalent.
\begin{enumerate}[\rm (a)]
\item $f$ has order $\geq m$ at $x_0$.
\item  $f(x) = O( \sfd(x_0,x)^m)$ for $x \to x_0$.
\end{enumerate}
\end{proposition}

Let $(z,U)$ be a privileged coordinate system and define $\delta_s:U \to U$ such that
$$\delta_s:(z_1, z_2, \dots, z_i, \dots, z_n) \mapsto (s^{w_1} z_1 , s^{w_2} z_2, \dots , s^{w_i} z_i, \dots, s^{w_n} z_n).$$
Relative to this privileged coordinate system, define functions, forms and vector fields of weight $m$ or $\geq m$ as in Section~\ref{sec:WeightC}. We warn the reader not to confuse notions of \emph{weight} and \emph{order}. In particular, for an arbitrary function $f,$ being of order $m$ at $x_0$ is an intrinsic property of $f$ and $D$ while being homogenous of weight $m$ is something that depends on the chosen privileged coordinate system. These notions are however related in the following way, see \cite[Section~5.1]{Bel96}.
\begin{lemma} \label{lemma:Orderweight}
A function $f$ is of order $\geq m$ at $x_0$ if and only if it has weight $\geq m$ relative to any privileged coordinate system centered at $x_0$. In particular, $f$ has order $m$ at $x_0$ if and only if it has weight $\geq m$, but not weight $\geq m+1$.
\end{lemma} 
Lemma~\ref{lemma:Orderweight} also gives us the following corollary.
\begin{corollary}
Relative to a privileged coordinate system $(z,U)$ centered at $x_0$, if a function, differential form or vector field has weight $\geq m$ (resp. not weight $\geq m$), then the same holds true relative to any other privileged coordinate system.
\end{corollary}
\begin{proof}
The statement is true for functions by Lemma~\ref{lemma:Orderweight}. In particular, the function $z_i$ must have weight $\geq w_i$ but not weight $\geq w_i +1$ with respect to any privileged coordinate system. The same must then hold true for the form $dz_i$. This gives us the desired result for forms since, near $x_0$, all such forms can be written as a sum of elements $f dz_{i_1} \wedge \cdots \wedge dz_{i_q}$ for some function $f$. Finally, any vector field $X$ is of weight $\geq m$ if and only if for any one-form $\alpha$ of weight $ \geq m_2$ we have that $\alpha(X)$ is of weight $\geq m + m_2$, which shows independence of privileged coordinate system for vector fields as well.
\end{proof}
Relative to a privileged coordinate system, we can define a radial gauge as in Lemma~\ref{lemma:RadialG}.
\begin{theorem}  \label{th:Curved}
Let $\pi:P \to M$ be a principal bundle over a sub-Riemannian manifold $(M, D, g)$, equipped with a connection $\omega$. Let $x_0 \in M$ be a regular point of $(M,D, g)$. Let $(z,U)$ be any privileged coordinate system around $x_0$. For any $p \in P_{x_0}$, let $\sigma: U \to P|_U$ be the corresponding radial gauge satisfying $\sigma(x_0) = p$. Define
$$F_p(\gamma) = \int_{\disk(\gamma)} \sigma^* \Omega.$$
Assume that $\sigma^* \Omega$ has weight $\geq m_\Omega$ at $x_0$. Then
$$\Hol^\omega_p = \exp\left( - F_p + O(\ell^{2m_\Omega}) \right).$$
\end{theorem}
We will give the proof in Section~\ref{sec:extended}.

\subsection{Tangent cone of a sub-Riemannian space at a regular point} \label{sec:dilation}
Let $x_0$ be a given regular point $(M, D, g)$. By choosing an appropriate neighborhood $U$ of $x_0$, we can make the following assumptions.
\begin{enumerate}[\rm (i)]
\item All points in $U$ are regular. Hence, $D|U$ is equiregular with growth vector $(n_1, \dots, n_\lambda)$.
\item $D|_U$ has a global orthonormal basis $X_1, \dots, X_{n_1}$.
\item There is a privileged coordinate system $z: U \to \mathbb{R}^n$ centered at $x_0$.
\end{enumerate}
For now on, we will replace $M$ with $U$ and consider the sub-Riemannian manifold $(U, D, g)$. We will use the privileged coordinates $(z_1, \dots, z_n)$ to identify $U$ with $\mathbb{R}^n$. In particular, we identify $x_0$ with $0$.

Define $Z_1, \dots, Z_{n_1}$ as the $-1$ homogeneous component of our orthonormal basis, meaning that
$$Z_j(z) := \lim_{s\downarrow 0} s (\delta_s^* X_j)(z), \qquad z \in U.$$
Let $\mathfrak{n} = \Lie \{ Z_1, \dots, Z_{n_1} \}$ be the Lie algebra generated by these vector fields. Since these vector fields are homogeneous of weight $-1$ and since there are no nonzero homogeneous vector fields of weight $< -\lambda$, the Lie algebra $\mathfrak{n}$ must be nilpotent of step at most $\lambda$ by \eqref{BracketDelta}. Actually, from \cite{Bel96}, we can conclude the following.
\begin{theorem} 
\begin{enumerate}[\rm (a)]
\item The vector fields $Z_1, \dots, Z_{n_1}$ are linearly independent at every point.
\item Define a sub-Riemannian structure on $(\hat D, \hat g)$ such that $Z_1, \dots, Z_{n_1}$ is an orthonormal basis. Then $\hat D$ and $D$ have the same growth vector. Furthermore, $(U, \hat D, \hat g)$ is a sub-Riemannian Carnot group of step $\lambda$, $\mathfrak{n}$ is its Lie algebra and $(z_1, \dots, z_n)$ are exponential coordinates with respect to some basis of $\mathfrak{n}$.
\item If $\hat \sfd$ is the Carnot-Carath\'eodory metric of $(\hat D, \hat g)$, then $(U, \hat \sfd)$ is the metric tangent cone of $(M, \mathsf{d})$ at $x_0$ in the sense of Gromov-Hausdorff convergence.
\item For any $\rho > 0$ there exists an $\eps > 0$ such that
$$(1- \rho) \hat \sfd(0, x) \leq \sfd(0, x) \leq (1+ \rho) \hat \sfd(0, x).$$
whenever $\sfd(0,x) < \eps$.
\end{enumerate}
\end{theorem}

\subsection{Dilation of extended space and proof of Theorem \ref{th:Curved}} \label{sec:extended}
We continue with the notation and assumptions of Section \ref{sec:dilation}.
Introduce a space $\tilde U  = \mathbb{R}_{> 0} \times U$ with coordinates $(s,z) = (s, z_1, \dots z_n)$. Define a projection $\nu: \tilde U \to U$ by $(s,z) \mapsto z$. Consider any vector field $X$ on $U$ with local flow $e^{tX}$ as a vector field on $\tilde U$ by
$$(Xf)(s,z) = \left. \frac{d}{dt} f(s, e^{tX}(z)) \right|_{t=0}.$$
Introduce a map $\tilde \delta_s: \tilde U \to \tilde U$ by
$$\tilde \delta_s(s_0, z_0) = (s s_0, \delta_s(z)).$$
Then $\nu \circ \tilde \delta_s = \delta_s \circ \nu$, and so $\tilde \delta_s^* \nu^* \alpha = \nu^* \delta_s^* \alpha$ for any form on $U$.

Relative to the global orthonormal basis $X_1, \dots, X_{n_1}$ of $D$, define vector fields $\tilde X_1, \dots, \tilde X_{n_1}$ by
$$\tilde X_j(s,z) := s^{-1}  (\tilde \delta_{1/s}^* X_j)(s,z).$$
Define a sub-Riemannian structure $(\tilde D, \tilde g)$ on $\tilde U$ such that the vector fields $\tilde X_1, \dots \tilde X_{n_1}$ form an orthonormal basis. The subbundle $\tilde D$ is then not bracket-generating, but we will still consider the Carnot-Carath\'eodory metric $\tilde \sfd$ of $(\tilde D, \tilde g)$. By definition, any $\tilde D$-horizontal curve is on the form $\tilde \gamma(t) = (s, \delta_{s}\gamma(t))$ where $\gamma(t)$ is a $D$-horizontal curve in $U$ and $s >0$.
Furthermore, if $\gamma(t)$ has finite length, then
$$ \ell(\tilde \gamma) := \int_0^1 \left| \dot {\tilde \gamma}(t) \right|_{\tilde g} \, dt = s \ell(\gamma).$$
As a consequence,
$$\tilde \sfd\big( (s_1, z_1), (s_2, z_2) \big) = \left\{ \begin{array}{ll} \infty & \text{if } s_1 \neq s_2, \\
s_1 \sfd(z_1, z_2) & \text{if } s_1 = s_2. \end{array}\right.$$
The manifold $U$ is isometric to the submanifold $\{ (1, z) \in \tilde U \, : \, z \in U\}$ and we will identify these two manifolds from now on.

\begin{lemma} \label{lemma:Curved}
Let $\alpha$ be a two-form on $U$ that is of weight $\geq m$ taking values in some finite dimensional vector space $\mathfrak{g}$. Consider the map
$$\Phi^\alpha: \scrL(0,U) \to \mathfrak{g}, \qquad \gamma \mapsto \int_{\disk(\gamma)} \alpha.$$
Then $\Phi^{\alpha}= O(\ell^m)$.
\end{lemma}
\begin{proof}
It is sufficient to consider the case $\mathfrak{g} = \mathbb{R}$. We consider $\scrL(0,U)$ as a subset $\scrL(\tilde U)$. Define $\tilde \Phi: \scrL(\tilde U) \to \mathbb{R}$ by $\tilde \Phi(\tilde \gamma) = \Phi^\alpha(\nu(\gamma))$. Consider any loop $\tilde \gamma(t) = (s_0, \delta_{s_0}\gamma(t))$ where $\gamma$ is a horizontal loop in $U$ of length $\ell(\gamma)$ and $s_0 \geq 1$. Assume that $\ell(\tilde \gamma) \leq \eps$ so that  $s_0 \ell(\gamma) \leq \eps$. Furthermore, choose $\eps$ sufficiently small such that $\frac{1}{2} \hat \sfd(0, x) \leq \sfd(0,x) \leq \frac{3}{2}\hat \sfd(0,x)$ whenever $\sfd(0,x) < 2\eps$.

Write $\iota_S\alpha = \sum_{i=1}^n f_j dz_i$ and define $\beta_s$ as
$$\beta_s = \sum_{j=1}^n s^{w_j} \left(\int_0^1 r^{w_j-1} (\delta_{sr}^*f_j) dr \right) dz_j.$$
By possibly choosing $\eps$ smaller, we can assume that
$$| f_j|(x) \leq C_j \sfd(0,x)^{m-1} \leq (\frac{3}{2})^{m-1} C_j \hat \sfd(0,x)^{m-1},$$
whenever $x < \eps$. Since $\sfd(0, \delta_{s_0} \gamma(t)) < \eps$, we have
\begin{enumerate}[\rm (i)]
\item $\tilde \Phi (\tilde \delta_s \tilde \gamma) = \Phi^\alpha(\delta_{s_0 s} \gamma) = \int_0^1 \beta_{s_0 s} (\dot \gamma(t)) dt$,
\item $| \Phi^\alpha(\delta_{s_0 s} \gamma)| \leq \ell(\gamma) \int_0^1 | \beta_{s_0 s} |_{g^*}(\gamma(t)) \, dt$,
\item If $C_0 = \max_{1 \leq j \leq n_1}  \{ C_j\} $, then for $0 < s \leq 1$,
\begin{align*} |\beta_{s_0s}|_{g^*}(x) & \leq \frac{3^{m-1} n_1 s_0^m s^m}{2^{m-1}}  C_0 \left(\int_0^1 r^{m-1} dr \right) \hat \sfd(0,x)^{m-1} \\
&  \leq \frac{3^{m-1} n_1 s_0^m s^m}{m}  C_0 \sfd(0,x)^{m-1}.
\end{align*}
As a consequence, $\|\Phi^\alpha(\delta_{s_0 s} \gamma)\| \leq \frac{3n_1 C_0 \eps^m }{2^{m-1} m} s^m$. \end{enumerate}
Hence, $\|\tilde \Phi( \delta_s \tilde \gamma) \| \leq \frac{3n_1 C_0 \eps^m }{2^{m-1} m} s^m$ and since $\tilde U$ a metric space with dilations, $\tilde \Phi = O(\ell^m)$ by Lemma~\ref{lemma:StoL} . As a consequence, we finally have $\Phi = O(\ell^m)$.
\end{proof}

\begin{proof}[Proof of Theorem~\ref{th:Curved}] 
Define $U$ and $\tilde U$ as above. Consider the pull-back bundle $\nu^* \pi: \nu^* P \to \tilde U$ and define $\pr_P: \nu^*P \to P$ by $(s,z, p) \mapsto p$ for $(s,z) \in \tilde U$ and $p \in P$. This gives us a pull-back connection $\tilde \omega = \pr_P^* \omega$ on $\nu^* \pi$, which has curvature $\pr_P^* \Omega$.

We may now modify the proof of Theorem~\ref{th:Flat} as in the proof of Lemma~\ref{lemma:Curved} for the result.
\end{proof}

\begin{remark}
The gauge free expression of $F^k_p$ found in Section \ref{sec:GaugeFree} can be used in Theorem \ref{th:Curved} as well.
\end{remark}

\begin{remark} \label{re:NotSurj}
Assume that $(z, U)$ is a coordinate system satisfying all conditions of a privileged coordinate system centered at $x_0$ except that $z(U) = V\subseteq \mathbb{R}^n$ where $V$ is some neighborhood of $0$. Define $\delta_s: \mathbb{R}^n \to \mathbb{R}^n$ by $\delta_s(z_1, \dots, z_n) = (sz_1, \dots, s^{w_i} z_i, \dots, s^{\lambda} z_n)$.
Consider the norm, $\|z\| = \sum_{i=1}^n |z_i|^{1/w_i}$, and define
$$V_r = \{ z\in \mathbb{R}^n \, : \, \| z \| \leq 1 \}.$$
Let $\eps > 0$ be such that $V_\eps \subseteq V$ and define $U^\vee = z^{-1} (V_\eps)$. We can then for $s \leq 1$, define dilations $\delta_s^\vee: U^\vee \to U^\vee$. Finally, taking a diffeomorphism $\phi: V_\eps \to \mathbb{R}^n$ such that $\phi| V_{\eps/2} = \id$, it follows that we can define a proper privileged coordinate system $\tilde z = \phi \circ z| U^\vee$ whose dilations equals $\delta^\vee_s$ for points close to $x_0$.

Hence, if we have a coordinate system $(z,U)$ which in privileged in all ways apart from being surjective, we can still define Taylor expansions as in Section~\ref{sec:GaugeFree} with these coordinates and Theorem~\ref{th:Curved} is still valid.
\end{remark}
\section{Increase of order in the sub-Riemannian case} \label{sec:Equivalent}
If $\sfd$ is the Carnot-Carath\'eodory metric of a sub-Riemannian structure~$(D, g)$, any curve of finite length is $D$-horizontal. Hence, when considering the expansion of holonomy in terms of length, the following concept is natural to consider. If $P \to M$ is a principal bundle and~$D$ is a subbundle of~$TM$ then two connections $\omega$ and $\tilde \omega$ are called \emph{$D$-equivalent} if $\omega(v) = \tilde \omega(v)$ for any $v \in TP$ with $\pi_* v \in D$. Let $[\omega]_D$ denote the equivalence class of $\omega$. By definition, $\Hol^\omega_p(\gamma) = \Hol^{\tilde \omega}_p(\gamma)$ for any $D$-horizontal loop $\gamma$ based at $x \in M$, $p \in P_x$ and $\tilde \omega \in [\omega]_D$.

Assume that $D$ is an equiregular subbundle with corresponding flag $D = D^1 \subseteq D^2 \subseteq \cdots \subseteq D^\lambda$. Denote by $D^0$ the zero section of $TM$. Let $\Ann(D^j)$ denote the subbundle of $T^*M$ of covectors vanishing on $D^j$. 
\begin{definition}
A two-vector-valued one-form $\chi\in \Gamma(T^*M \otimes \bigwedge^2 TM)$ is called a \emph{selector} of $D$ if it satisfies the following two properties.
\begin{enumerate}[\rm (i)]
\item For $0 \leq j \leq r-1$, $\chi(D^{j+1}) \subseteq \bigwedge^{2} D^j$.
\item For $1 \leq j \leq r-1$ and any $\alpha \in \Gamma(\Ann(D^j))$,
$$v \mapsto \alpha(v) + d\alpha(\chi(v))$$
vanishes on $D^{j+1}$.
\end{enumerate}
\end{definition}
Relative to such a selector $\chi$, the corresponding contraction operator 
$$\iota_\chi : \bigwedge^2 T^*M \to T^*M$$
is defined by $(\iota_\chi \eta)(v) = \eta(\chi(v))$, $v \in T_x M$, $\eta \in \bigwedge^2 T^*_x M$, $x \in M$.
It was shown in \cite{CGJK15} that any equiregular subbundle $D$ has at least one selector. Actually, from the construction in [3, Lemma~2.7], we know that $D$ has a selector such that $\chi(D^j) \subseteq \bigoplus_{i=1}^{j-1} D^i \oplus D^{j-i}$ for any $j =2, \dots, \lambda$. From the same reference, we also have the following relation between holonomy along $D$-horizontal curves and selectors.
\begin{theorem} \label{th:HorHol}
Let $\chi$ be a selector of $D\subseteq TM$. Let $\omega$ be an arbitrary connection on a principal bundle $P \to M$. Then there exists a unique connection $\tilde \omega \in [\omega]_D$ such that its curvature $\tilde \Omega$ satisfies
$$\iota_\chi \tilde \Omega = 0.$$
Furthermore, for any $x \in M$ and $p \in P_x,$ this connection satisfies
$$\left\{ \Hol^{\tilde \omega}_p(\gamma) \, : \begin{array}{c} \text{$\gamma$ is $D$-horizontal} \\ \gamma(0) = \gamma(1) = x \end{array} \right\} = \left\{ \Hol^{\tilde \omega}_p(\gamma) \, : \, \gamma \in \scrL(x) \right\}.$$
Finally, if $\Omega$ is the curvature of $\omega$, then
\begin{eqnarray} \label{tildeomega} \tilde \omega &=& \omega +  \iota_\chi \sum_{j=1}^{\lambda-1} \binom{\lambda-1}{j} (L^\omega \iota_\chi)^{j-1} \Omega, \\
\label{tildeOmega} \tilde \Omega &=& (\id + L^\omega \iota_\chi)^{\lambda-1} \Omega,\end{eqnarray}
where $\lambda$ is the step of $D$ and $L^\omega: \Gamma(T^*M \otimes \Ad(P)) \to \Gamma(\bigwedge^2 T^*M \otimes \Ad(P))$ is given by
\begin{equation} \label{Lomega} L^\omega \alpha = d^{\nabla^\omega} \alpha + \frac{1}{2} [\alpha, \alpha].
\end{equation}
\end{theorem}
We see that a choice of selector gives us a unique way of choosing a connection in a $D$-equivalence class, which is ``optimal'', in the sense that the holonomy of any loop $\gamma$ is also the holonomy of some $D$-horizontal curve. Recall that the exterior covariant derivative $d^{\nabla^\omega}$ of $\Ad(P)$-valued $j$-forms is determined by the following rules,
\begin{enumerate}[\rm (i)]
\item For $j = 0$, any $d^{\nabla^\omega} \varphi = \nabla^\omega \varphi$ for any $\varphi \in \Gamma(\Ad(P))$,
\item If $\alpha$ is an $\Ad(P)$-valued $j$-form and $\beta$ a real-valued form
$$d^{\nabla^\omega} (\alpha \wedge \beta) = d^{\nabla^\omega} \alpha \wedge \beta + (-1)^j \alpha \wedge d\beta.$$
\end{enumerate}

If the increase in the growth vector is maximal, then we can improve our expansion by changing our connection to a $D$-equivalent one. Given $n_1 = \mathrm{rank}(D)$, the maximal increase in the growth vector $(n_1,\ldots,n_\lambda)$ is given by the free nilpotent Lie
algebra. The free nilpotent Lie algebra $\mathfrak{f}[n_1; k]$ of step $k$ with $n_1$ generators is defined as the free Lie algebra on $n_1$ generators divided by the ideal generated by brackets of length greater than $k$. If we define $\nu[n_1; k]= \mathrm{rank}(\mathfrak{f}[n_1; k])$, then
$$n_k \leq\nu[n_1;\lambda].$$
Remark in particular that $n_2 \leq\nu[n_1; 2]=\frac{1}{2}n_1(n_1 -1)$.

\begin{corollary}
Let $(M, D, g)$ be a sub-Riemannian manifold with $D$ equiregular of step $\lambda$, with canonical flag $D \subseteq D^2 \subseteq \cdots \subseteq D^\lambda =TM$ and with growth vector $(n_1, \dots, n_\lambda)$. Let $\omega$ be some connection on $\pi: P \to M$.
\begin{enumerate}[\rm (a)]
\item Assume that $\tilde \Omega$ defined as in \eqref{tildeOmega} vanishes on $\sum_{i=1}^{k-1} D^i \oplus D^{k-i}$. Then 
$$\Hol_p^\omega = \exp \left( - \tilde F^{2k+1}_p + O(\ell^{2k+2}) \right),$$
where $\tilde F^j_p(\gamma) = \int_{\disk(\gamma)} \Tay_p^{\tilde \omega}(\tilde \Omega; j)$, $\tilde \omega$ is defined as in \eqref{tildeomega} and $\Tay_p^{\tilde \omega} (\tilde \Omega; j)$ is defined relative to some privileged coordinate system centered at $x \in M$.
\item Assume that $n_k = \nu[n_1,k]$. Let $\chi$ be a selector satisfying
$$\chi(D^j) \subseteq \bigoplus_{i=1}^{j-1} D^i \oplus D^{j-i}, \qquad j=2, \dots, k,$$
and define $\tilde \Omega$ as \eqref{tildeOmega} relative to $\chi$. Then $\tilde \Omega$ vanish on $\sum_{i=1}^{k-1} D^i \oplus D^{k-i}$.
\end{enumerate}
\end{corollary}

\begin{proof}
The first statement follows by Theorem~\ref{th:HorHol} and from the fact that if $\tilde \Omega$ vanish on $\oplus_{i=1}^{k-1} D^i \oplus D^{k-i}$, it must have weight $\geq k+1$. For the second statement, observe that $n_j = \nu[n_1,j]$ for any $j =1, \dots, k$. Let $hX$ denote the horizontal lift of a vector field on $M$ to $P$ with respect to $\tilde \omega$. Write $\mathcal{E} = \ker \tilde \omega$, $\mathcal{E}^j = \{ w \in \mathcal{E} \, : \, \pi_* w  \in D^j \}$ and $\underline{\mathcal{E}}^j = \Gamma(\mathcal{E}^j)$. Note that $\rank \mathcal{E}^j = n_j$. If we then define $\underline{\mathcal{F}}^1 = \underline{\mathcal{E}}^1$ and $\underline{\mathcal{F}}^{j+1} = \mathcal{F}^j +[\underline{\mathcal{F}}^{1}, \underline{\mathcal{F}}^{j}]$, we claim that $\underline{\mathcal{E}}^j = \underline{\mathcal{F}}^j$ for any $i=1, \dots, k$.

We will first show that  $\underline{\mathcal{E}}^j = \underline{\mathcal{F}}^j$ for any $j =1, \dots, \lambda$ by induction. Assume that $\underline{\mathcal{E}}^j = \underline{\mathcal{F}}^j$ for some $j$. For any vector field $Z \in\Gamma(D^{j+1})$ with $\chi(Z) = \sum_{i=1}^l X_i \wedge Y_i $ and for any $\alpha$, we have $\alpha(Z - \sum_{i=1}^l [X_i, Y_i]) = \alpha(Z) + d\alpha(\chi(Z)) = 0$. This means that
$$Z = Z_0 + \sum_{i=1}^l [X_i, Y_i] ,$$
for some $Z_0 \in \Gamma(D^j)$. By our assumptions $X_i \wedge Y_i$ also take values in $\bigoplus_{i=1}^j D^i \oplus D^{j+1-i}$. Since we know that $\tilde \Omega(\chi(Z)) = 0 = \sum_{i=1}^l \Omega(X_i, Y_i)$, we have that
$$hZ = hZ_0 + \sum_{i=1}^l [hX_i, hY_i] \in \underline{\mathcal{F}}^{j+1} .$$
This completes the induction step.

As a consequence of the above statement,
$$\mathcal{E}_x^j \subseteq \mathcal{F}^j_x = \{ Y(x) \, : \, Y \in \underline{F}^j \} \subseteq T_x M,$$
for any $x \in M$. However, since $\mathcal{E}^1$ has rank $n_1$, we must have $\rank \mathcal{F}^j_x \leq \nu[n_1,k] = \rank \mathcal{E}^j$ for $j = 1, \dots, k$ and so $\underline{\mathcal{E}}^j = \underline{\mathcal{F}}^j$. This means that for any vector field $X \in \Gamma(D^i)$ and $Y \in \Gamma(D^{k-i})$, we have
$$[hX, hY] \in \Gamma(\mathcal{E}^k),$$
which is equivalent to the curvature vanishing on $\sum_{i=1}^k D^i \oplus D^{k-i}$.

\end{proof}

\section{Examples} \label{sec:Examples}

\subsection{ Euclidean space} \label{sec:Rn}
Consider the special case when $N$ is an inner product space. Let $x_0 \in N$ be an arbitrary point. We may choose coordinates $(z_1, \dots, z_n)$ so that we can identify $N$ with $\mathbb{R}^n$ with the standard inner product and $x_0 =0$. Write $\partial_k = \frac{\partial}{\partial z_k}$. Let $\pi : P \to \mathbb{R}^n$ be a principal bundle with connection $\omega$ and curvature $\Omega$. Since any $\Omega$ has $m_\Omega \geq 2$, we have
$$\Hol_p^\omega = \exp\left(- F^3_p + O(\ell^4) \right),$$
with $F_p^k$ as in Theorem~\ref{th:Taylor} and $p \in P_0$. Write $\Omega = \frac{1}{2} \sum_{i,j=1}^n \Omega_{ij} dz_i \wedge dz_j$ with $\Omega_{ij} \in \Gamma(\Ad(P))$, $\Omega_{ij} = - \Omega_{ji}$. Then
\begin{align*}
F_p^3(\gamma) & = \frac{1}{2} \sum_{i,j=1}^n \int_{\disk \gamma} \left( \Omega_{ij}(p) + \sum_{k=1}^n z_k \nabla_{\partial_k}^\omega \Omega_{ij}(p) \right) dz_i \wedge dz_j \\
& = \frac{1}{2} \sum_{i,j=1}^n \Omega_{ij}(p) \int_{\gamma}  z_i dz_j + \frac{1}{3} \sum_{i,j,k=1}^n \nabla_{\partial_k}^\omega \Omega_{ij}(p) \int_\gamma z_i z_k dz_j .
\end{align*}
If $(\sigma, U)$ is any gauge around $0$ such that $\sigma(0) = p$, and we write $(\sigma^* \omega)(\partial_k) = \omega^\sigma_k$ and $\sigma^* \Omega_{ij} = \Omega_{ij}^\sigma$, then
\begin{align} \label{RnF3}
F_p^3(\gamma) & = \frac{1}{2} \sum_{i,j=1}^n \Omega_{ij}^\sigma(0) \int_{\gamma}  z_i dz_j \\ \nonumber
& \quad + \frac{1}{3} \sum_{i,j,k=1}^n \left( (\partial_k \Omega_{ij}^\sigma)(0) + [ \omega_k^\sigma(0), \Omega_{ij}^\sigma(0) ]\right) \int_\gamma z_i z_k dz_j .
\end{align}

\subsection{Riemannian manifolds} For any $n$-dimensional Riemannian manifold $(M,g)$ and point $x_0 \in M$, choose a normal coordinate system $(z,U)$ centered at $x_0$. This coordinate system satisfies all assumptions of privileged coordinate system except surjectivity. It follows from Remark~\ref{re:NotSurj} that for any $p \in P_{x_0}$, $\Hol_p^\omega = \exp\left(- F^3_p + O(\ell^4) \right)$ with $F_p^3$ written as in \eqref{RnF3} in the coordinates $(z,U)$.

\subsection{The Heisenberg group} \label{sec:Heisenberg}
Consider the $3$-dimensional Heisenberg group $N$. This is the Carnot group of step~$2$ with Lie algebra
$$\mathfrak{n} = \mathfrak{n}_1 \oplus \mathfrak{n}_2, \qquad \mathfrak{n}_1 = \spn \{ Z_1, Z_2\}, \qquad \mathfrak{n}_2 = \spn\{ Z_3\}.$$
where $\mathfrak{n}_2$ is the center and $[Z_1, Z_2] = Z_3$.
Denote the left invariant vector fields corresponding to $Z_1$, $Z_2$, $Z_3$ by the same symbol.
If we give the point $\exp(xZ_1 + y Z_2 +z Z_3)$ coordinates $(x, y, z)$ then
$$Z_1 = \frac{\partial}{\partial x} - \frac{1}{2} y \frac{\partial}{\partial z} , \qquad Z_2 = \frac{\partial}{\partial y} + \frac{1}{2} x \frac{\partial}{\partial z}, \qquad Z_3 = \frac{\partial}{\partial z}.$$
The corresponding coframe is given by $dx$, $dy$ and $\theta = dz + \frac{1}{2} (y dx - x dy)$.

We want to compute our approximation of holonomy on $N$. Using left translation, it is sufficient to consider loops based at $0$. We will use Theorem~\ref{th:HorHol}. The unique selector $\chi$ of $D$ is given by
$$\chi: Z_1 \mapsto 0, \qquad Z_2 \mapsto 0, \qquad Z_3 \mapsto Z_1 \wedge Z_2.$$
The corresponding contraction operator is
$$\iota_\chi : dx \wedge dy \mapsto \theta, \qquad dx \wedge \theta \mapsto 0, \qquad dy \wedge \theta \mapsto 0.$$
Write $\Omega = \Omega^1 dx \wedge \theta + \Omega^2 dy \wedge \theta + \Omega^3 dx \wedge dy$. Define
\begin{align*} \tilde \Omega & = \Omega + d^{\nabla^\omega} \Omega^3 \wedge \theta = \left(\Omega^1 + \nabla^\omega_{Z_1} \Omega^3 \right) dx \wedge \theta + \left( \Omega^1 + \nabla^\omega_{Z_2} \Omega^3\right) dy \wedge \theta \\
& =: \Psi^1 dx \wedge \theta + \Psi^2 dy \wedge \theta,
\end{align*}
and $\tilde \omega = \omega + \beta$, where
$$\beta = \iota_\chi \Omega = \Omega^3 \theta.$$
Notice that for any $\varphi \in \Gamma(\Ad(P))$ and vector field $X$, we have $\nabla^{\tilde \omega}_X \varphi = \nabla_X^\omega \varphi + [\beta(X), \varphi]$.

We now want to compute $\Tay_p^{\tilde \omega}(\tilde \Omega; 5) = \Tay_p^\omega(\Psi^1;2) dx\wedge \theta + \Tay_p^\omega(\Psi^2 ; 2) dy \wedge \theta$. With $\bnabla^{\tilde \omega}_\alpha$ defined as in Section~\ref{sec:WEF} and for $j = 1,2$,
\begin{align*} & \Tay_p(\Psi^j, 2) \\
& = \Psi^j(p) + z_1 \bnabla^{\tilde \omega}_1\Psi^j (p) + z_2 \bnabla^{\tilde \omega}_2\Psi^j (p) + z_3 \bnabla^{\tilde \omega}_3\Psi^j (p) \\
&\quad +\frac{1}{2} \left( z_1^2 \bnabla^{\tilde \omega}_{1,1} \Psi^j (p) + z_2^2 \bnabla^{\tilde \omega}_{2,2} \Psi^j (p)  \right) + \frac{z_1 z_2}{2} \left(  \bnabla^{\tilde \omega}_{1,2} \Psi^j (p) + z_2^2 \bnabla^{\tilde \omega}_{2,1} \Psi^j (p)  \right) \\
& = \Psi^j(p) + z_1 \bnabla^{\omega}_1\Psi^j (p) + z_2 \bnabla^{\omega}_2\Psi^j (p) + z_3 \left( \bnabla^{\omega}_3\Psi^j (p) + [\Omega^3, \Psi^j] \right) \\
& \quad +\frac{1}{2} \left( z_1^2 \bnabla^{\omega}_{1,1} \Psi^j (p) + z_2^2 \bnabla^{\omega}_{2,2} \Psi^j (p)  \right) + \frac{z_1 z_2}{2} \left(  \bnabla^{\omega}_{1,2} \Psi^j (p) + \textcolor{red}{z_2^2}\bnabla^{ \omega}_{2,1} \Psi^j (p)  \right) 
\end{align*} 

which equals
\begin{align*} & \Tay_p(\Psi^j, 2) \\
& = \Omega^j(p) + z_1 \bnabla^{\omega}_1\Omega^j (p) + z_2 \bnabla^{\omega}_2\Omega^j (p) + z_3 \left( \bnabla^{\omega}_3 \Omega^j (p) + [\Omega^3,  \Omega^j] \right) \\
& \quad +\frac{1}{2} \left( z_1^2 \bnabla^{\omega}_{1,1}  \Omega^j (p) + z_2^2 \bnabla^{\omega}_{2,2}  \Omega^j (p)  \right) + \frac{z_1 z_2}{2} \left(  \bnabla^{\omega}_{1,2} \Omega^j (p) + z_2^2 \bnabla^{ \omega}_{2,1}  \Omega^j (p)  \right) \\
& \quad +  \bnabla^{\omega}_j \Omega^3(p) + z_1 \bnabla^{\omega}_{1,j} \Omega^3 (p) + z_2 \bnabla^{\omega}_{2,j} \Omega^3 (p) + z_3 \left( \bnabla^{\omega}_{3, j} \Omega^3 (p) + [\Omega^3,  \bnabla^{\omega}_j\Omega^3] \right) \\
& \quad +\frac{1}{2} \left( z_1^2 \bnabla^{\omega}_{1,1,j}  \Omega^3 (p) + z_2^2 \bnabla^{\omega}_{2,2,j}  \Omega^3 (p)  \right) + \frac{z_1 z_2}{2} \left(  \bnabla^{\omega}_{1,2,j} \Omega^3 (p) + \bnabla^{ \omega}_{2,1,j}  \Omega^3 (p)  \right) .
\end{align*} 
In conclusion, we obtain
$$\Hol^\omega_p = \exp \left( - F_p^5 + O(\ell^6)\right),$$
where $F^5_p = \int_{\disk(\gamma)} \Tay_p^{\tilde \omega} (\tilde \Omega;5) = \int_\gamma q^1(x \theta - 2z dx) + \int_\gamma q^2 (y \theta - 2z dy)$, where
\begin{align*} q^j(z) & = \textcolor{red}{\frac{1}{3}} \Omega^j(p) +\frac{z_1}{4} \bnabla^{\omega}_1\Omega^j (p) + \frac{z_2}{4}  \bnabla^{\omega}_2\Omega^j (p) + \frac{z_3}{5} \left( \bnabla^{\omega}_3 \Omega^j (p) + [\Omega^3,  \Omega^j] \right) \\
& \quad +\frac{1}{10} \left( z_1^2 \bnabla^{\omega}_{1,1}  \Omega^j (p) + z_2^2 \bnabla^{\omega}_{2,2}  \Omega^j (p)  \right) + \frac{z_1 z_2}{10} \left(  \bnabla^{\omega}_{1,2} \Omega^j (p) + z_2^2 \bnabla^{ \omega}_{2,1}  \Omega^j (p)  \right) \\
& \quad + \frac{1}{3} \bnabla^{\omega}_j \Omega^3(p) + \frac{z_1}{4} \bnabla^{\omega}_{1,j} \Omega^3 (p) + \frac{z_2}{4} \bnabla^{\omega}_{2,j} \Omega^3 (p) + \frac{z_3}{5} \left( \bnabla^{\omega}_{3, j} \Omega^3 (p) + [\Omega^3,  \bnabla^{\omega}_j\Omega^3] \right) \\
& \quad +\frac{1}{10} \left( z_1^2 \bnabla^{\omega}_{1,1,j}  \Omega^3 (p) + z_2^2 \bnabla^{\omega}_{2,2,j}  \Omega^3 (p)  \right) + \frac{z_1 z_2}{10} \left(  \bnabla^{\omega}_{1,2,j} \Omega^3 (p) + \bnabla^{ \omega}_{2,1,j}  \Omega^3 (p)  \right) .
\end{align*}

\subsection{Sub-Riemannian Hopf fibration}  \label{sec:Hopf} 
Consider the Lie algebra $\mathfrak{su}(2)$ given by basis $X_1$, $X_2$ and $X_3$ and relations
$$[X_1,X_2] = X_3, \qquad [X_2,X_3] = X_1, \qquad [X_3,X_1] = X_2.$$
This can be considered as the Lie algebra of the group of matrices $\mathrm{SU}(2)$ of the form
$$a=\left( \begin{array}{cc} w_1 & w_2 \\ - \bar{w}_2 & \bar{w}_1 \end{array} \right) , \qquad w_1, w_2 \in \mathbb{C}, \quad |w_1|^2 + |w_2|^2 = 1.$$
In these coordinates and with $w_j = u_j + i v_j$, $j =1,2$, we can identify $X_1$, $X_2$ and $X_3$ with the left invariant vector fields
\begin{eqnarray*}
2X_1 & = & -u_2 \partial_{u_1} - v_2 \partial_{v_1} + u_1 \partial_{u_2} + v_1 \partial_{v_2}, \\
2X_2 & = & -v_2 \partial_{u_1} + u_2 \partial_{v_1} - v_1 \partial_{u_2} + u_1 \partial_{v_2}, \\
2X_3 & = & -v_1 \partial_{u_1} + u_1 \partial_{v_1} + v_2 \partial_{u_2} - u_2 \partial_{v_2} .
\end{eqnarray*}
Introduce a sub-Riemannian structure $(D,g)$ on $\mathrm{SU}(2)$ by defining $X_1$ and $X_2$ to be an orthonormal basis.

On the neighborhood $U = \{ a \in \mathrm{SU}(2) \, : \, u_1 > 0 \}$, define coordinates
$$z: U \to \mathbb{R}^3, \qquad a \mapsto \left(2u_2, 2v_2, 2v_1  \right).$$
This is a privileged coordinate system centered at the identity $(u_1, v_1, u_2, w_2) = (1,0,0,0)$ apart from the requirement of surjectivity.
Furthermore, the dilations $\delta_s:U\to U$ are well defined for $0 < s \leq 1$. 

The vector fields $X_1$, $X_2$, $X_3$ in the coordinate system $z$ are given by
\begin{eqnarray*}
X_1 & = & z_1 Y + \sqrt{1-|z|^2} \partial_{z_1} + z_3 \partial_{z_2} - z_2 \partial_{z_3}, \\
X_2 & = & z_2 Y - z_3 \partial_{z_1} + \sqrt{1-|z|^2} \partial_{z_2} +z_1 \partial_{z_3} , \\
X_3 & = & z_3 Y + z_2 \partial_{z_1} - z_1 \partial_{z_2} + \sqrt{1-|z|^2} \partial_{z_3} . \\
Y & = & \frac{1}{\sqrt{1-|z|^2}} \left( z_1 \partial_{z_1} + z_2 \partial_{z_2} + z_3 \partial_{z_3} \right)
\end{eqnarray*}

Define $\alpha_1, \alpha_2, \alpha_3$ as the coframe dual to $X_1, X_2, X_3$. In the coordinates $(z_1, z_2, z_3)$ this is given by
\begin{eqnarray*}
\alpha_1 & = &  \sqrt{1-|z|^2} dz_1 + z_3 dz_2 - z_2 dz_3 , \\
\alpha_2 & = & - z_3 dz_1 + \sqrt{1-|z|^2} dz_2 +z_1 dz_3  \\
\alpha_3 & = &  z_2 dz_1 - z_1 dz_2 + \sqrt{1-|z|^2} dz_3 . 
\end{eqnarray*}

Let $P \to \mathrm{SU}(2)$ be a principal bundle with connection $\omega$. We study the holonomy of horizontal loops based at $1$.
Let $\Omega$ be the curvature of $\omega$, given in the coframe by
$\Omega = \Omega^1 \alpha_1 \wedge \alpha_3 + \Omega^2 \alpha_2 \wedge \alpha_3 + \Omega^3 \alpha_1 \wedge \alpha_2$.
with $\Omega^j \in \Gamma(\Ad P)$. The unique selector $\chi: T^*M \to \bigwedge^2 T^*M$ is given by
$$\chi: X_1 \mapsto 0, \qquad X_2 \mapsto 0, \qquad X_3 \mapsto X_1\wedge X_2.$$
This allows us to us to use the connection $\tilde \omega = \omega + \Omega^3 \alpha^1$ and curvature
$$\tilde \Omega = (\Omega^1 + \nabla^\omega_{X_1} \Omega^3) \alpha_1 \wedge \alpha_3 + (\Omega^2 + \nabla^\omega_{X_2} \Omega^3) \alpha_2 \wedge \alpha_3 =: \Psi^1 \alpha_1 \wedge \alpha_3 + \Psi^2 \alpha_2 \wedge \alpha_3, $$
which has order $\geq 3$.

For $j = 1,2$, consider the vector fields $Z_j = \lim_{s \downarrow 0} s \delta^*_s X_j $ and define $Z_3 = [Z_1, Z_2] = \lim_{s \downarrow 0} s^2 \delta_s^* X_3$. Then
$$ Z_1 =  \partial_{z_1}  -  z_2 \partial_{z_3}  ,\quad Z_2  =  \partial_{z_2} + z_1 \partial_{z_3}, \quad Z_3  = \partial_{z_3} .$$
Furthermore, $S = z_1 Z_1+ z_2 Z_2 + 2z_3 Z_3.$
Let $\bnabla$ be the connection such that $Z_1$, $Z_2$, $Z_3$ are parallel. The corresponding frame is given by $dz_1$, $dz_2$ and $\theta = dz_3 + z_2 dz_1 - z_1 dz_2$. It follows that
$$\tilde \Omega = \Xi^1 dz_1 \wedge \theta+ \Xi^2 dz_2 \wedge \theta + \Xi^3 dz_1 \wedge dz_2,$$
with
\begin{eqnarray*}
\Xi^1 & = & \Psi^1  (1-z_1^2 - z_3^2) - \Psi^2  ( z_3\sqrt{1-|z|^2} + z_1 z_2),  \\
\Xi^2 & = & \Psi^1 (z_3  \sqrt{1-|z|^2} - z_2z_1) + \Psi^2 (  1-z_2^2 - z_3^2 ). \\
\Xi^3 & =& - \Psi^1 (1- \sqrt{1- |z|^2}) ( z_1\sqrt{1-|z|^2} + z_2 z_3 ) \\
& & + \Psi^2 ( 1- \sqrt{1-|z|^2}) ( z_1 z_3  - z_2 \sqrt{1-|z|^2} ), 
\end{eqnarray*}
Hence, $\Hol^\omega_p =  \exp\left( - F_p^5 + O(\ell^6) \right)$ where
\begin{align*} F^5_p(\gamma) & = \int_{\disk(\gamma)} \Tay_p^{\tilde \omega}(\Xi^1; 2) dz_1 \wedge \theta + \int_{\disk(\gamma)} \Tay_p^{\tilde \omega}(\Xi^2; 2) dz_2 \wedge \theta \\
& \quad + \int_{\disk(\gamma)} \Tay_p^{\tilde \omega}(\Xi^3; 3) dz_1 \wedge dz_2.
\end{align*}

%

\appendix
\section{Approximation of ODEs in Lie groups}

\subsection{Lie groups and right logarithmic derivatives}
Let $G$ be a connected Lie group with Lie algebra $\mathfrak{g}$.
For any continuous curve $[0, t_1] \to \mathfrak{g}$, $t \mapsto A(t)$, consider the initial value problem on $G$,
$$\dot a(t) = A(t) \cdot a(t), \qquad a(0) = 1.$$
Assume that for any $t \in [0, t_1]$ the curve $Q(t) = \exp^{-1} a(t)$ in $\mathfrak{g}$ is well defined. Note that 
$$d\exp:T_A \mathfrak{g} \to T_{\exp A} G, \qquad B \mapsto \exp(A) \cdot f(\ad A) B, \qquad f(z) = \frac{1 -e^{-z}}{z}.$$
Hence
$$f(\ad Q(t)) \dot Q(t) = A(t),\qquad Q(0) = 0,$$
and
\begin{equation} \label{EqQ} \dot Q(t) = g(\ad Q(t)) A(t), \qquad g(z) = \frac{z}{1- e^{-z}}. \end{equation}
Consider $\mathfrak{g}$ as a subalgebra of $\mathfrak{gl}(q,\mathbb{C})$ for some $q >0$ and let $\| \blank \|$ be a Banach algebra norm on $\mathfrak{gl}(q, \mathbb{C})$. Write $\| A\|_{L^\infty[0,t]}  = \sup_{s \in [0,t]} \| A(s)\|$.
\begin{lemma} \label{lemma:Picard}
Define constants $B_1$ and $B_2$ such that on the disk $\{ z \in \mathbb{C} \, : \, |z| \leq \pi \}$,
$$| g(z)| \leq B_1, \qquad |g'(z)| \leq B_2.$$
Define
$$\| \ad \| = \sup \{ \| \ad(A_1) A_2 \| \, : \, \|A_1\| = \|A_2\| = 1 \}.$$
Then there exists $\eps > 0$, such that for any $A:[0,T] \to \mathfrak{g}$ with
$$t \| A\|_{L^\infty[0,t]} < \eps,$$
we have 
$$\left\| Q(t) - \int_0^t A(s) \,ds \right\|    \leq  \frac{ \| \ad \| t^2 B_2^2}{1-  \| \ad \| B_2 \eps} \| A\|_{L^\infty[0,t]}^2.$$
\end{lemma}

\begin{proof}
We want to estimate $Q(t)$ using Picard iteration. 
For any $t > 0$, write $M_t = \sup_{s \in [0,t]} \| A(s)\|$. 
 Define $\mathscr{A}(Q)(t) = \int_0^t g(\ad Q(s)) A(s) \, ds$ and write $b  = \pi/\|\ad\|$. Note that
$$\|\mathscr{A}(Q)\|_{L^\infty[0,t]} \leq t M_t B_1  \text{ whenever } \|Q\|_{L^\infty[0,t]} \leq b,$$
and
$$\|\mathscr{A}(Q_1 - Q_2)\|_{L^\infty[0,t]} \leq t M_t B_2 \| \ad\|  \|Q_1 - Q_2\|_{L^\infty[0,t]} .$$

Define $B = \max \{ B_1 \|\ad \|/\pi , B_2\| \ad\| \}$ and assume that $t M_t < 1/B$. Then it follows that $Q(t) = \lim_{n \to \infty} \mathscr{A}^n(0)$ is the solution of \eqref{EqQ}
and
\begin{align*} & \| Q(t) - \mathscr{A}(0)(t) \| \leq \| Q - \mathscr{A}(0) \|_{L^\infty[0,t]} \\
& \leq  tM_t B_2 \| \ad\| \| Q(t)\|_{L^\infty [0,t]} \leq  \frac{t^2 M_t^2 B_2^2}{1- t M_t B_2 }  .\end{align*}
The result follows from choosing $\eps \leq 1/B$.

\end{proof}


\bibliographystyle{habbrv}
\bibliography{Bibliography}

\end{document}